\documentclass[10pt,leqno]{amsart}
\theoremstyle{definition}
\usepackage{indentfirst, amssymb,amsmath,amsthm}
\usepackage{csquotes,color}
\newtheorem*{theoA}{Theorem A}
\newtheorem*{theoB}{Theorem B}
\newtheorem*{theoC}{Theorem C}
\newtheorem*{theoD}{Theorem D}

\newtheorem{theo}{Theorem}[section]
\newtheorem{lem}{Lemma}[section]

\newtheorem{exm}{Example}[section]
\newtheorem{defi}{Definition}[section]
\newtheorem{rem}{Remark}[section]
\newtheorem{question}{Question}[section]
\newcommand{\ol}{\overline}
\newcommand{\be}{\begin{equation}}
\newcommand{\ee}{\end{equation}}
\newcommand{\beas}{\begin{eqnarray*}}
\newcommand{\eeas}{\end{eqnarray*}}
\newcommand{\bea}{\begin{eqnarray}}
\newcommand{\eea}{\end{eqnarray}}

\numberwithin{equation}{section}
\begin{document}
\title[A note on uniqueness of meromorphic functions and its derivative\\
  ]{A note on uniqueness of meromorphic functions  and its derivative sharing two sets }
\date{}
\author[A. Banerjee and B. Chakraborty ]{ Abhijit Banerjee  and Bikash Chakraborty }
\date{}
\address{$^{1}$Department of Mathematics, University of Kalyani, West Bengal 741235, India.}
\email{abanerjee\_kal@yahoo.co.in, abanerjeekal@gmail.com, abanerjee\_kal@rediffmail.com}
\address{$^{2}$Department of Mathematics, University of Kalyani, West Bengal 741235, India.}
\address{$^{2}$Present Address: Department of Mathematics, Ramakrishna Mission Vivekananda
Centenary College, West Bengal, India 700118.}
\email{bikashchakraborty.math@yahoo.com, bikashchakrabortyy@gmail.com}
\maketitle
\let\thefootnote\relax
\setcounter{footnote}{0}

\begin{abstract} In this paper, on the basis of a specific question raised in (\cite{3.1}),  we further continue our investigations on the uniqueness of a meromorphic function with its higher derivatives sharing two sets and answered the question affirmatively. We exhibited some examples to show the sharpness of some conditions used in our main result.
\end{abstract}
\section{Introduction and Definitions}
By $\mathbb{C}$ and $\mathbb{N}$, we mean the set of complex numbers and set of positive integers respectively. We also we assume that readers are familiar with the classical Nevanlinna theory (\cite{8}).\par
In 1976, Gross (\cite{7}) first generalized the concept of value sharing problem by proposing his famous question on set sharing. To understand Gross' contribution elaborately, we require the following definition of set sharing:
\begin{defi}
For a non constant meromorphic function $f$ and any set $S\subset \mathbb{C}\cup\{\infty\}$, we define
$$E_{f}(S)=\bigcup_{a \in S}\{(z,p) \in \mathbb{C}\times\mathbb{N}~ |~ f(z)=a ~with~ multiplicity~ p\},$$
$$\ol{E}_{f}(S)=\bigcup_{a \in S}\{z \in \mathbb{C}~ |~ f(z)=a, counting ~without~ multiplicity\}$$
Two meromorphic functions $f$ and $g$ are said to share the set $S$ counting multiplicities (CM), if $E_{f}(S)=E_{g}(S)$. They are said to share the set $S$ ignoring multiplicities (IM), if $\ol{E}_{f}(S)=\ol{E}_{g}(S).$\par
Thus if $S$ is a singleton set, then the set sharing notion is nothing but the value sharing notion.
\end{defi}
In 1976, Gross (\cite{7}, Question 6) proposed a problem concerning the uniqueness of entire functions that share sets of distinct elements instead of values as follows:\\
{\bf Question A:} {\it Can one find a finite set $S$ such that any two non constant entire functions $f$ and $g$ satisfying $E_{f}(S)=E_{g}(S)$  must be identical ?}\par
In this directions, there are many elegant results in the literature but in the present scenario our prime focus will be on Gross' following question which deal with the two sets sharing problems:\\
{\bf Question B:} {\it Can one find two finite set $S_{j}$ for $j=1,2$  such that any two non constant entire functions $f$ and $g$ satisfying $E_{f}(S_{j})=E_{g}(S_{j})$ for $j=1,2$  must be identical ?}\par
In (\cite{7}), Gross also asked: \enquote{If the answer to \emph{Question 6} is affirmative, it would be interesting
to know how large both sets would have to be.}\par
An affirmative answer of the above question was provided by Yi \cite{13} et al. Since then, shared sets problems have been studied by many authors and a number of profound results have been obtained.\par
Taking the question of Gross (\cite{7}) into the background, the following question is natural:\\
{\bf Question C:} (\cite{13.1}, \cite{13.2}, \cite{12.1}) {\it Can one find two finite sets $S_{j}$ for $j=1,2$ such that any two non constant meromorphic functions $f$ and $g$ satisfying $E_{f}(S_{j})=E_{g}(S_{j})$ for $j=1,2$  must be identical ?}\par
In connection to the above question, a brief survey can be found in (\cite{3.1}). In this context, the possible best result is due to Prof. H. X. Yi (\cite{13.1}).\par In 2002, Yi (\cite{13.1}) proved that there exist two finite sets $S_1$ with one element and $S_2$ with eight elements such that any two non constant meromorphic functions $f$ and $g$ satisfying $E_{f}(S_{j})=E_{g}(S_{j})$ for $j=1,2$ must be identical.\par
In the mean time, Prof. I. Lahiri (\cite{9}) introduced a new notion, namely \enquote{weighted sharing} which is the scaling between  counting multiplicities and ignoring multiplicities. As far as relaxations of the nature of sharing of the sets are concerned, this notion has a remarkable influence.
\begin{defi} (\cite{9}) Let $k\in\mathbb{N}\cup\{0\}\cup\{\infty\}$. For $a\in\mathbb{C}\cup\{\infty\}$, we denote by $E_{k}(a;f)$ the set of all $a$-points of $f$, where an $a$-point of multiplicity $m$ is counted $m$ times if $m\leq k$ and $k+1$ times if $m>k$. If $E_{k}(a;f)=E_{k}(a;g)$, we say that $f$ and $g$ share the value $a$ with weight $k$.\par
Let $S$ be a set of distinct elements of $\mathbb{C}\cup\{\infty\}$. We denote by $E_{f}(S,k)$ the set $E_{f}(S,k)=\bigcup_{a\in S}E_{k}(a;f)$. If $E_{f}(S,k)=E_{g}(S,k)$, then we say that $f$ and $g$ share the set $S$ with weight $k$.
\end{defi}
In 2008, the present first author (\cite{2.1}) improved the result of Yi (\cite{13.1}) by relaxing the nature of sharing the range sets by the notion of weighted sharing. He established that there exist two finite sets $S_1$ with one element and $S_2$ with eight elements such that any two non constant meromorphic functions $f$ and $g$ satisfying $E_{f}(S_{1},0)=E_{g}(S_{1},0)$ and $E_{f}(S_{2},2)=E_{g}(S_{2},2)$  must be identical.\par
But after that no remarkable improvements were done regarding Question C. So the natural query would be whether one can get better result even for particular class of meromorphic functions. This possibility encouraged researchers to find the similar types range sets corresponding to the derivatives of two meromorphic functions. But in this particular direction results are scanty in number.\par
To proceed further, we first recall the existing results in this direction:
\begin{theoA} (\cite{12.1}) Let $S_{1}=\{z:z^{n}+az^{n-1}+b=0\}$ and $S_{2}=\{\infty\}$, where $a$, $b$ are nonzero constants such that $z^{n}+az^{n-1}+b=0$ has no repeated root and $n\;(\geq 7)$, $k$ be two positive integers. Let $f$ and $g$ be two non constant meromorphic functions  such that $E_{f^{(k)}}(S_{1},\infty)=E_{g^{(k)}}(S_{1},\infty)$ and $E_{f}(S_{2},\infty)=E_{g}(S_{2},\infty)$. Then $f^{(k)}\equiv g^{(k)}$.
\end{theoA}
In 2010, Banerjee-Bhattacharjee (\cite{3.2}) improved the above results in the following way:
\begin{theoB} (\cite{3.2})
Let $S_{i}$ ($i=1,2$) and $k$ be given as in {\em Theorem A}. Let $f$ and $g$ be two non constant meromorphic functions  such that either $E_{f^{(k)}}(S_{1},2)=E_{g^{(k)}}(S_{1},2)$ and $E_{f}(S_{2},1)=E_{g}(S_{2},1)$ or $E_{f^{(k)}}(S_{1},3)=E_{g^{(k)}}(S_{1},3)$ and $E_{f}(S_{2},0)=E_{g}(S_{2},0)$. Then $f^{(k)}\equiv g^{(k)}$.
\end{theoB}
In 2011, the same authors (\cite{3.3}) further improved the above results in the following manner:
\begin{theoC} (\cite{3.3})  Let $S_{i}$ ($i=1,2$) and $k$ be given as in {\em Theorem A}. Let $f$ and $g$ be two non constant meromorphic functions such that $E_{f^{(k)}}(S_{1},2)=E_{g^{(k)}}(S_{1},2)$ and $E_{f}(S_{2},0)=E_{g}(S_{2},0)$. Then $f^{(k)}\equiv g^{(k)}$.
\end{theoC}
Recently the present authors (\cite{3.1}) improved the above results at the cost of considering a new range sets instead of the previous. To discuss the results in (\cite{3.1}), we first require the following polynomial.\par
Suppose for two positive integers $m$, $n$
\bea\label{1abcp1}
P_{\ast}(z)=z^{n}-\frac{2n}{n-m}z^{n-m}+\frac{n}{n-2m}z^{n-2m}+c,
\eea
where $c$ is any complex number satisfying $|c|\not=\frac{2m^2}{(n-m)(n-2m)}$ and $c\not=0,-\frac{1-\frac{2n}{n-m}+\frac{n}{n-2m}}{2}$.\par
\begin{theoD}\label{sm1.4}(\cite{3.1}) Let $n(\geq 6),~m(=1)$ and $k\geq 1$ be three positive integers. Let $S_{\ast}=\{z :P_{\ast}(z)=0\}$ where the polynomial $P_{\ast}(z)$ is defined by  (\ref{1abcp1}). Let $f$ and $g$ be two non constant meromorphic functions satisfying $E_{f^{(k)}}(S_{\ast},3)=E_{g^{(k)}}(S_{\ast},3)$ and $E_{f^{(k)}}(0,0)=E_{g^{(k)}}(0,0)$. Then $f^{(k)} \equiv g^{(k)}$.
\end{theoD}
In (\cite{3.1}), the following question was asked:
\begin{question}
 \emph{Whether there exists two suitable sets $S_{1}$ (with one element) and $S_{2}$ (with five elements) such that when derivatives of any two non constant meromorphic functions share them with finite weight yield $f^{(k)}\equiv g^{(k)}$ ?}
\end{question}
To answer the above question affirmatively is the main aim of writing this paper.
To this end, we recall some definitions which we need in this sequel.
 \begin{defi} (\cite{3}) Let $z_{0}$ be a zero of $f-a$ and $g-a$ of multiplicity $p$ and $q$ respectively.\par
Then $N^{1)}_{E}(r,a;f)$ and $\ol N^{(2}_{E}(r,a;f)$ denote the reduced counting functions of those $a$-points of $f$ and $g$ where $p=q=1$ and $p=q\geq 2$ respectively.\par
Also $\ol N_{L}(r,a;f)$ and $\ol N_{L}(r,a;g)$ denote the reduced counting functions of those $a$-points of $f$ and $g$ where $p>q\geq 1$ and $q>p\geq 1$ respectively.\par
Thus clearly $N^{1)}_{E}(r,a;f)= N^{1)}_{E}(r,a;g)$ and $\ol N^{(2}_{E}(r,a;f)=\ol N^{(2}_{E}(r,a;g).$
\end{defi}
\begin{defi} (\cite{3}) Let $f$ and $g$ share a value $a$-IM. We denote by $\ol N_{*}(r,a;f,g)$ the reduced counting function of those $a$-points of $f$ whose multiplicities differ from the multiplicities of the corresponding $a$-points of $g$, i.e, $\ol N_{*}(r,a;f,g)=\ol N_{L}(r,a;f)+\ol N_{L}(r,a;g)$.
\end{defi}
\section{Main Result}
For a positive integer $n$, we shall denote by $P(z)$ the following polynomial (\cite{13}):
\bea\label{abcp1}
P(z)=z^{n}+az^{n-1}+b~~~\text{where}~~ab\not=0~\text{and}~\frac{b}{a^{n}}\not=(-1)^{n}\frac{(n-1)^{(n-1)}}{n^{n}}.
\eea
\begin{theo}\label{thB5}
Let $n(\geq 5)$ and $k(\geq 1)$ be two positive integers. Let $S=\{z :P(z)=0\}$ where the polynomial $P(z)$ is defined by  (\ref{abcp1}). Let $f$ and $g$ be two non constant meromorphic functions satisfying $E_{f^{(k)}}(S,2)=E_{g^{(k)}}(S,2)$ and $E_{f^{(k)}}(0,1)=E_{g^{(k)}}(0,1)$. Then $f^{(k)}\equiv g^{(k)}.$
\end{theo}
\begin{rem} Theorem \ref{thB5} shows that there exists two sets $S_{1}$ with one element and $S_{2}$ with five elements such that when derivatives of any two non constant meromorphic functions share them with finite weight yields $f^{(k)}\equiv g^{(k)}$. Thus the above theorem improves {\it Theorem D} in the direction of {\it Question 1.1}.
\end{rem}
As Theorem \ref{thB5} deals with specific class of meromorphic functions, so the general curiosity will be:
\begin{question} Does Theorem \ref{thB5} hold good for general class of meromorphic functions?
\end{question}
The following example shows that the answer is negative, i.e., $k\geq1$ is sharp.
\begin{exm}
Let $n \in \mathbb{N}$ and $S=\{z :P(z)=0\}$ where  $P(z)$ is defined by (\ref{abcp1}). We choose $f(z)=-a\frac{1-h^{n-1}}{1-h^{n}}$ and, $g(z)=h(z)f(z)$, where $h(z)=\frac{e^{z}-1+sgn(n-5)}{e^{z}+1}$ and $sgn(x)$ is defined as:
$$ sgn(x)=\bigg\{ \begin{array}{c}
+1~\text{for}~ x\geq0,\\
-1~\text{for}~x<0.
\end{array}
$$
For $n \in \mathbb{N}$, $\left(\frac{f}{g}\right)^{n-1}\left(\frac{f+a}{g+a}\right)=\frac{1}{h^{n-1}}\bigg\{\frac{a\frac{h^{n-1}-h^{n}}{1-h^{n}}}{a\frac{1-h}{1-h^{n}}}\bigg\}=1$. Thus  $E_{f}(S,\infty)=E_{g}(S,\infty)$.\par
Again
from the construction of $h$, it is clear that $0$ is a Picard exceptional value of $h$ only when  $n\geq 5$. So $g=hf$ implies $E_{f}(0,\infty)=E_{g}(0,\infty)$ only when $n\geq 5$. Here $f$ and $g$ satisfies all the conditions stated in Theorem \ref{thB5}, but $f\not\equiv g$.
\end{exm}
Obviously the next natural query would be:
\begin{question}
Whether the set $S$ can be replaced by any arbitrary set of five elements in the same environment of Theorem \ref{thB5} ?
\end{question}
The following example shows that the answer is negative.
\begin{exm}
Let $f(z) =\frac{1}{(\sqrt{\alpha \beta \gamma })^{k-1}}e^{\sqrt{\alpha \beta \gamma}\;z}$ and, $g(z)=\frac{(-1)^{k}}{(\sqrt{\alpha \beta \gamma})^{k-1}}e^{-\sqrt{\alpha \beta \gamma}\;z}$ ($k\geq 1$) and \linebreak $S=\{\alpha \sqrt{\beta },\alpha \sqrt{\gamma},\beta \sqrt{\gamma },\gamma \sqrt{\beta },\sqrt{\alpha\beta\gamma}\}$, where $\alpha $, $\beta $ and $\gamma $ are three nonzero distinct complex numbers. Clearly $E_{f^{(k)}}(S,\infty)=E_{g^{(k)}}(S,\infty)$ and $E_{f^{(k)}}(0,\infty)=E_{g^{(k)}}(0,\infty)$, but $f^{(k)}\not\equiv g^{(k)}$.
\end{exm}
Also the next two examples show that $ab\not=0$ is necessary in Theorem \ref{thB5}.
\begin{exm}
If $b=0$, then $S=\{0,-a\}$. Now let $f(z) =ae^{z}$ and, $g(z)=(-1)^{k}ae^{-z}$ ($k\geq 1$). Clearly $E_{f^{(k)}}(S,\infty)=E_{g^{(k)}}(S,\infty)$ and $E_{f^{(k)}}(0,\infty)=E_{g^{(k)}}(0,\infty)$, but $f^{(k)}\not\equiv g^{(k)}$.
\end{exm}
\begin{exm}
If $a=0$, then $S=\{z|z^{n}+b=0\}$ where $n\geq5$. Now let $f(z) =\lambda e^{z}$ and, $g(z)=(-1)^{k}\lambda e^{-z}$ ($k\geq 1$) where $\lambda$ is one of the value of $(-b)^{\frac{1}{n}}$. Clearly $E_{f^{(k)}}(S,\infty)=E_{g^{(k)}}(S,\infty)$ and $E_{f^{(k)}}(0,\infty)=E_{g^{(k)}}(0,\infty)$, but $f^{(k)}\not\equiv g^{(k)}$.
\end{exm}
The following example shows that when derivative of two meromorphic functions share two sets, if cardinality of one set is one, then cardinality of another set is atleast three.
\begin{exm}\label{abc2015.}
Let $S_{1}=\{a\}$ and $S_{2}=\{b,c\}$ $(a\not=b,a\not=c)$. Choose $f(z) =p(z)+(b-a)e^{z}$ and, $g(z)=q(z)+(-1)^{k}(c-a)e^{-z}$, where $p(z)=\sum\limits_{i=0}^{k-1}c_{i}z^{i}+\frac{a}{k!}z^{k}$ and $q(z)=\sum\limits_{i=0}^{k-1}d_{i}z^{i}+\frac{a}{k!}z^{k}$, $c_{i},d_{i}\in\mathbb{C}$.
Clearly $E_{f^{(k)}}(S_j,\infty)=E_{g^{(k)}}(S_j,\infty)$ for $j=1,2$, but $f^{(k)}\not\equiv g^{(k)}$.
\end{exm}
The following two examples show that if we choose different sets other than the specific form of choosing the first set $S$ with three or four elements  Theorem \ref{thB5} ceases to hold.
\begin{exm}\label{abc2015}
Choose four nonzero distinct complex numbers $\alpha, \beta, \gamma$ and $\delta$ such that $\alpha\beta=\gamma\delta$.\par
Let $f(z) =-\alpha e^{z}$ and, $g(z)=(-1)^{k+1}\beta e^{-z}$ ($k\geq 1$) and  $S=\{\alpha, \beta, \gamma, \delta\}$.\\ Clearly $E_{f^{(k)}}(S,\infty)=E_{g^{(k)}}(S,\infty)$ and $E_{f^{(k)}}(0,\infty)=E_{g^{(k)}}(0,\infty)$, but $f^{(k)}\not\equiv g^{(k)}$.
\end{exm}
\begin{exm}\label{abc20151}
Choose three nonzero distinct complex numbers $\alpha, \beta, \gamma$  such that $\alpha\beta=\gamma^{2}$.\par
Let $f(z) =-\alpha e^{z}$ and, $g(z)=(-1)^{k+1}\beta e^{-z}$ ($k\geq 1$) and  $S=\{\alpha, \beta, \gamma\}$.\\ Clearly $E_{f^{(k)}}(S,\infty)=E_{g^{(k)}}(S,\infty)$ and $E_{f^{(k)}}(0,\infty)=E_{g^{(k)}}(0,\infty)$, but $f^{(k)}\not\equiv g^{(k)}$.
\end{exm}
However the following question is inevitable from the above discussion:
\begin{question} Whether there exists two suitable sets $S_{1}$ with one element and $S_{2}$ with three or four elements such that when derivatives of any two non constant non entire meromorphic functions  share them with finite weight yield $f^{(k)}\equiv g^{(k)}$ ?
 \end{question}
\section {Lemmas}
$$\text{We define}~~~~F:=-\frac{(f^{(k)})^{n-1}(f^{(k)}+a)}{-b}~~\text{and}~~G:=-\frac{(g^{(k)})^{n-1}(g^{(k)}+a)}{-b},$$  where $n(\geq 1)$ and $k(\geq 1)$ are integers. By $H$ and $\Phi$, we mean the following two functions $$H:=\left(\frac{F^{''}}{F^{'}}-\frac{2F^{'}}{F-1}\right)-\left(\frac{G^{''}}{G^{'}}-\frac{2G^{'}}{G-1}\right)~~\text{and}~~\Phi:=\frac{F^{'}}{F-1}-\frac{G^{'}}{G-1}.$$
$$\text{Also define}~~T(r):=\max\{T(r,f^{(k)}),~ T(r,g^{(k)})\}~~\text{and}~~S(r)=o(T(r)).$$
\begin{lem}\label{bd2}  If $f^{(k)}$ and $g^{(k)}$ share $(0,0)$, then $F\equiv G$ gives $f^{(k)}\equiv g^{(k)}$ where $k, n \in \mathbb{N}$.
\end{lem}
\begin{proof}
The inequality  $\ol{N}(r,\infty;f^{(k)})\leq\frac{1}{1+k}N(r,\infty;f^{(k)})\leq \frac{1}{2}N(r,\infty;f^{(k)})\leq \frac{1}{2}T(r,f^{(k)})+O(1)$
implies $\Theta(\infty,f^{(k)})\geq\frac{1}{2}$  for  $k\geq 1$. Again $F\equiv G$ implies $f^{(k)}$ and $g^{(k)}$ share $(\infty,\infty)$. Now rest of the proof is similar as \emph{Lemma 1} of (\cite{bbb}). So we omit the details.
\end{proof}
\begin{lem}\label{bb3} Suppose $f^{(k)}$ and $g^{(k)}$ share $(0,0)$, then $FG\not\equiv 1$ for $k\geq 1$ and $n\geq3$.
\end{lem}
\begin{proof} On contrary, we assume that $FG\equiv 1$. Then
\bea\label{r0} (f^{(k)})^{n-1}\left(f^{(k)}+a\right)(g^{(k)})^{n-1}\left(g^{(k)}+a\right)\equiv b^{2}.\eea
So by Mokhon'ko's Lemma (\cite{11.1}), we have $T(r,f^{(k)})=T(r,g^{(k)})+O(1)$. If $z_{0}$ be a $(-a)$-point of $f^{(k)}$ of order $p$, then $z_{0}$ is a pole of $g$ of order $q$ such that $p=n(q+k)\geq n$.
So $$\ol{N}(r,-a;f^{(k)})\leq\frac{1}{n}N(r,-a;f^{(k)}).$$
Again from (\ref{r0}), we have $\ol{N}(r,0;f^{(k)})=\ol{N}(r,0;g^{(k)})=S(r)$ as $E_{f^{(k)}}(0,0)=E_{g^{(k)}}(0,0)$.\par
Now by the Second Fundamental Theorem, we get
 \bea\label{r1} T(r,f^{(k)}) &\leq& \ol{N}(r,\infty;f^{(k)})+\ol{N}(r,0;f^{(k)})+\ol{N}(r,-a;f^{(k)})+S(r,f^{(k)})\\
\nonumber  &\leq& \ol{N}(r,-a;g^{(k)})+\ol{N}(r,-a;f^{(k)})+S(r)\\
\nonumber  &\leq& \frac{2}{n}T(r,f^{(k)})+S(r),\eea
  which is a contradiction as $n\geq 3$.
\end{proof}
\begin{lem}\label{b4}(\cite{3.11})
If $F$ and $G$ share $(1,l)$ where $0\leq l<\infty$, then\\
$\ol{N}(r,1;F)+\ol{N}(r,1;G)-N_{E}^{1)}(r,1,F)+(l-\frac{1}{2})\ol{N}_{*}(r,1;F,G)\leq\frac{1}{2}(N(r,1;F)+N(r,1;G)).$
\end{lem}
\begin{lem}\label{b7} Let $F$, $G$ and $\Phi$ be defined as previously and $F\not\equiv G$. If $f^{(k)}$ and $g^{(k)}$ share $(0,q)$ where $0\leq q<\infty$ and $F$, $G$ share $(1,l)$, then
\beas & &\{(n-1)q+n-2\}\;\ol N(r,0;f^{(k)}\mid\geq q+1)\\
&\leq& \ol{N}(r,\infty;f^{(k)})+\ol{N}(r,\infty;g^{(k)})+\ol{N}_{*}(r,1;F,G)+S(r),\eeas
for $n(\geq3)\in\mathbb{N}$. Similar expressions hold for $g^{(k)}$ also. \end{lem}
\begin{proof} The proof is similar to the proof of \emph{Lemma 3.6} of (\cite{3.1}). So we omit the details.
\end{proof}
\begin{lem}\label{today} If $f^{(k)}$ and $g^{(k)}$ share $(0,0)$; $F$ and $G$ share $(1,2)$; $n(\geq3)\in\mathbb{N}$, then
$$\ol{N}_{*}(r,1;F,G) \leq \frac{n}{4n-10}\big\{\ol{N}(r,\infty;f^{(k)})+\ol{N}(r,\infty;g^{(k)})\big\}+S(r).$$
\end{lem}
\begin{proof}
Applying Lemma \ref{b7}, we have
\bea\label{noon1}  && 2\ol{N}_{*}(r,1;F,G) \leq 2\ol{N}(r,1;F|\geq 3)\\
\nonumber &\leq& N(r,0;f^{(k+1)}|f^{(k)}\not=0)+S(r)\\
\nonumber &\leq& \ol{N}(r,0;f^{(k)})+\ol{N}(r,\infty;f^{(k)})+S(r)\\
\nonumber &\leq& \frac{1}{n-2}\big\{\ol{N}(r,\infty;f^{(k)})+\ol{N}(r,\infty;g^{(k)})+\ol{N}_{*}(r,1;F,G)\big\}+\ol{N}(r,\infty;f^{(k)})+S(r).\eea
Again,
\bea\label{noon2}  && 2\ol{N}_{*}(r,1;F,G)\leq 2\ol{N}(r,1;G|\geq 3)\\
\nonumber &\leq& \frac{1}{n-2}\big\{\ol{N}(r,\infty;g^{(k)})+\ol{N}(r,\infty;f^{(k)})+\ol{N}_{*}(r,1;F,G)\big\}+\ol{N}(r,\infty;g^{(k)})+S(r).\eea
Adding (\ref{noon1}) and (\ref{noon2}), we have
\bea\label{noon3}  && \ol{N}_{*}(r,1;F,G)\leq \frac{n}{4n-10}\big\{\ol{N}(r,\infty;f^{(k)})+\ol{N}(r,\infty;g^{(k)})\big\}+S(r).\eea
Hence the proof is completed.
\end{proof}
\begin{lem}\label{biku1} Suppose $f$ and $g$ are two non constant meromorphic functions such that $f^{(k)}$ and $g^{(k)}$ share $(S,l)$ and $(0,q)$, where $k\geq 1$, $l\geq 2$, 
 $n\geq 5$ and $q\geq 1,$
then $H \equiv 0$.
\end{lem}
\begin{proof} On contrary, we assume that $H \not\equiv 0$. Then clearly $F\not\equiv G$ and
\bea\label{m1}
&& \overline{N}(r,1;F|=1)=\overline{N}(r,1;G|=1)\leq N(r,\infty;H)\\
\nonumber&\leq& \ol{N}(r,\infty;f^{(k)})+\ol{N}(r,\infty;g^{(k)})+\ol{N}\left(r,-a\frac{n-1}{n};f^{(k)}\right)+\ol{N}\left(r,-a\frac{n-1}{n};g^{(k)}\right)\\
\nonumber&+& \ol{N}_{*}(r,0;f^{(k)},g^{(k)})+ \ol{N}_{*}(r,1;F,G)+\ol{N}_{0}(r,0;f^{(k+1)})+\ol{N}_{0}(r,0;g^{(k+1)}),\eea
where $\ol{N}_{0}(r,0;f^{(k+1)})$ is the reduced counting function of zeros of $f^{(k+1)}$ which is not zeros of $f^{(k)}\left(f^{(k)}+a\frac{n-1}{n}\right)$ and $(F-1)$. Similarly $\ol{N}_{0}(r,0;g^{(k+1)})$ is defined.\par
Now using the Second Fundamental Theorem, Lemma \ref{b4} and inequality (\ref{m1}), we get
\bea\label{n2}&& (n+1)(T(r,f^{(k)})+T(r,g^{(k)}))\\
\nonumber &\leq& \overline{N}(r,\infty;f^{(k)})+\overline{N}(r,0;f^{(k)})+\overline{N}(r,\infty;g^{(k)})+\overline{N}(r,0;g^{(k)})\\
\nonumber &+& \overline{N}(r,1;F)+\overline{N}(r,1;G)+\ol{N}\left(r,-a\frac{n-1}{n};f^{(k)}\right)+\ol{N}\left(r,-a\frac{n-1}{n};g^{(k)}\right)\\
\nonumber&-& N_{0}(r,0,f^{(k+1)})-N_{0}(r,0,g^{(k+1)})+S(r,f^{(k)})+S(r,g^{(k)})\\
\nonumber &\leq& 2\left\{\overline{N}(r,\infty;f^{(k)})+\overline{N}(r,\infty;g^{(k)})+\ol{N}\left(r,-a\frac{n-1}{n};f^{(k)}\right)+\ol{N}\left(r,-a\frac{n-1}{n};g^{(k)}\right)\right\}\\
\nonumber &+& \overline{N}(r,0;f^{(k)})+\overline{N}(r,0;g^{(k)})+\ol{N}_{*}(r,0;f^{(k)},g^{(k)})+\overline{N}(r,1;F)+\overline{N}(r,1;G)\\
\nonumber &-&\overline{N}(r,1;F|=1)+\ol{N}_{*}(r,1;F,G)+S(r,f^{(k)})+S(r,g^{(k)})\\
\nonumber &\leq& 2\big\{T(r,f^{(k)})+T(r,g^{(k)})\big\}+ 2\big\{\overline{N}(r,\infty;f^{(k)})+\overline{N}(r,\infty;g^{(k)})\big\}+\overline{N}(r,0;f^{(k)})\\
\nonumber &+& \overline{N}(r,0;g^{(k)})+\overline{N}(r,0;f^{(k)}|\geq q+1)+\frac{1}{2}(N(r,1;F)+N(r,1;G))\\
\nonumber &+& \left(\frac{3}{2}-l\right)\ol{N}_{*}(r,1;F,G)+S(r,f^{(k)})+S(r,g^{(k)})\\
\nonumber &\leq& \left(2+\frac{n}{2}\right)\big\{T(r,f^{(k)})+T(r,g^{(k)})\big\}+2\big\{\overline{N}(r,\infty;f^{(k)})+\overline{N}(r,\infty;g^{(k)})\big\}+2\overline{N}(r,0;f^{(k)})\\
\nonumber &+& \overline{N}(r,0;f^{(k)}|\geq q+1)+\left(\frac{3}{2}-l\right)\ol{N}_{*}(r,1;F,G)+S(r,f^{(k)})+S(r,g^{(k)}).\eea
Using Lemma \ref{b7} and Lemma \ref{today} in the above inequality, we obtain
\bea\label{n3}&&\left(\frac{n}{2}-1\right)(T(r,f^{(k)})+T(r,g^{(k)}))\\
\nonumber &\leq& 2\big\{\overline{N}(r,\infty;f^{(k)})+\overline{N}(r,\infty;g^{(k)})\big\}\\
\nonumber &+& \left(\frac{2}{n-2}+\frac{1}{(n-1)q+(n-2)}\right)\big\{\overline{N}(r,\infty;f^{(k)})+\overline{N}(r,\infty;g^{(k)})+\ol{N}_{*}(r,1;F,G)\big\}\\
\nonumber &+& \left(\frac{3}{2}-l\right)\ol{N}_{*}(r,1;F,G)+S(r,f^{(k)})+S(r,g^{(k)})\nonumber\\ &\leq& \left(2+\frac{2}{n-2}+\frac{1}{(n-1)q+(n-2)}\right)\big\{\ol{N}(r,\infty;f^{(k)})+\ol{N}(r,\infty;g^{(k)})\big\}\nonumber\\
\nonumber &+& \left(\frac{n}{4n-10}\right)\left(\frac{3}{2}-l+\frac{2}{n-2}+\frac{1}{(n-1)q+(n-2)}\right)\{\ol{N}(r,\infty;f^{(k)})+\ol{N}(r,\infty;g^{(k)})\}\\
\nonumber &+& S(r,f^{(k)})+S(r,g^{(k)}).\nonumber\eea
So from (\ref{n3}), we get
\bea&&\left(\frac{n}{2}-1\right)(T(r,f^{(k)})+T(r,g^{(k)}))\\
\nonumber &\leq& \left(1+\frac{1}{n-2}+\frac{1}{2(n-1)q+2(n-2)}\right)\{N(r,\infty;f^{(k)})+N(r,\infty;g^{(k)})\}\\
\nonumber &+& \left(\frac{n}{4n-10}\right)\left(\frac{1}{n-2}+\frac{1}{2(n-1)q+2(n-2)}-\frac{1}{4}\right)\big\{{N}(r,\infty;f^{(k)})+{N}(r,\infty;g^{(k)})\big\}\\
\nonumber &+& S(r,f^{(k)})+S(r,g^{(k)}),\eea
which leads to a contradiction when $n\geq 5$ and $q\geq 1$. Thus $H \equiv 0$.
\end{proof}
\begin{lem}\label{biku} Suppose $f^{(k)}$ and $g^{(k)}$ share $(0,0)$ and $H\equiv 0$. If $n\geq 4$ and $k\geq 1$, then $f^{(k)}\equiv g^{(k)}$.
\end{lem}
\begin{proof} Given $H\equiv 0$. On integration, we have
\bea\label{pe1.1} F\equiv\frac{AG+B}{CG+D},\eea
where $A,B,C,D$ are constant satisfying $AD-BC\neq 0 $. So $A=C=0$ never occur.\par
Thus clearly $F$ and $G$ share $(1,\infty)$. Now by Mokhon'ko's Lemma (\cite{11.1}), we have
\bea\label{pe1.2} T(r,f^{(k)})=T(r,g^{(k)})+S(r).\eea
Next we consider the following two cases:\\
\textbf{Case-1} Assume $AC\neq0.$ In this case (\ref{pe1.1}) can be written as
\bea\label{biral} F-\frac{A}{C}\equiv\frac{BC-AD}{C(CG+D)}.\eea
Now by using the Second Fundamental Theorem, equations (\ref{pe1.2}) and (\ref{biral}), we get
\beas && n T(r,f^{(k)})+O(1)=T(r,F)\\
 &\leq& \overline{N}(r,\infty;F)+\overline{N}(r,0;F)+\overline{N}\left(r,\frac{A}{C};F\right)+S(r,F)\\
&\leq& \overline{N}(r,\infty;f^{(k)})+\overline{N}(r,0;f^{(k)})+T(r,f^{(k)})+\overline{N}(r,\infty;g^{(k)})+S(r,f^{(k)})\\
&\leq& \frac{1}{2}\left(N(r,\infty;f^{(k)})+N(r,\infty;g^{(k)})\right)+\overline{N}(r,0;f^{(k)})+T(r,f^{(k)})+S(r,f^{(k)})\\
&\leq& 3T(r,f^{(k)})+S(r,f^{(k)}),\eeas
which is a contradiction as $n\geq 4$.\\
\textbf{Case-2} Next we assume $AC=0$. Now we consider the following subcases:\par
\textbf{Subcase-2.1} Let $A=0$ and $C\neq0$. Hence $B\neq0$. Thus equation (\ref{pe1.1}) becomes
\bea \label{chegg} F\equiv\frac{1}{\gamma G+\delta},~~\text{where}~~\gamma=\frac{C}{B}~~\text{and}~~\delta=\frac{D}{B}.\eea
If $F$ has no $1$-point,  then in view of the Second Fundamental Theorem, we get
\beas && n T(r,f^{(k)})+O(1)=T(r,F)\\
&\leq& \overline{N}(r,\infty;F)+\overline{N}(r,0;F)+\overline{N}(r,1;F)+S(r,F)\\
&\leq& \overline{N}(r,\infty;f^{(k)})+\overline{N}(r,0;f^{(k)})+T(r,f^{(k)})+S(r,f^{(k)})\\
&\leq& \frac{5}{2}T(r,f^{(k)})+S(r,f^{(k)}),\eeas
which is a contradiction as $n\geq 4$.\\
Thus there exist atleast one $z_{0}$ such that $F(z_{0})=G(z_{0})=1$. Hence from equation (\ref{chegg}), we get $\gamma+\delta=1$, $\gamma\neq0$ and hence
\bea F\equiv\frac{1}{\gamma G+1-\gamma}.\eea
If $\gamma\neq1$, then the Second Fundamental Theorem yields
\beas T(r,G) &\leq& \overline{N}(r,\infty;G)+\overline{N}(r,0;G)+\overline{N}\left(r,0;G+\frac{1-\gamma}{\gamma}\right)+S(r,G)\\
&\leq& \overline{N}(r,\infty;g^{(k)})+\overline{N}(r,0;g^{(k)})+T(r,g^{(k)})+\overline{N}(r,\infty;f^{(k)})+S(r,g^{(k)})\\
&\leq& \frac{3}{n}T(r,F)+S(r,F),\eeas
which is impossible as $n\geq 4$. Thus $\gamma=1$, i.e., $FG\equiv 1$ which is not possible by Lemma \ref{bb3}.\par
\textbf{Subcase-2.2} Let $A\neq0$ and $C=0$. Hence $D\neq0$. So we can write equation (\ref{pe1.1}) as
\bea F\equiv\lambda G+\mu,~~~\text{where}~~\lambda=\frac{A}{D}~~\text{and}~~\mu=\frac{B}{D}.\eea
If $F$ has no $1$-point, then proceeding as above in \emph{Subcase-2.1}, we arrive at a contradiction. Thus $\lambda+\mu=1$ with $\lambda\neq0$. If $\lambda \neq1$, then $B\not= 0$ and hence equation (\ref{pe1.1}) yields \beas \ol{N}(r,0,f^{(k)})=\ol{N}(r,0,g^{(k)})=S(r).\eeas
Now  applying the Second Fundamental Theorem, we obtain
\beas T(r,G) &\leq& \overline{N}(r,\infty;G)+\overline{N}(r,0;G)+\overline{N}\left(r,0;G+\frac{1-\lambda}{\lambda}\right)+S(r,G)\\
&\leq& \overline{N}(r,\infty;g^{(k)})+\overline{N}(r,0;g^{(k)})+T(r,g^{(k)})+\overline{N}(r,0;f^{(k)})+T(r,f^{(k)})+S(r,g^{(k)})\\
&\leq& \frac{5}{2n}T(r,G)+S(r,G),\eeas
which is a contradiction as $n\geq4$.\\
Thus $\lambda=1$ and hence $F\equiv G$. Consequently Lemma \ref{bd2} gives $f^{(k)}\equiv g^{(k)}$.
\end{proof}
\section {Proof of the theorem}
\begin{proof}[\textbf{Proof of Theorem \ref{thB5} }]
Given that $f^{(k)}$ and $g^{(k)}$ share $(S,2)$ and $(0,1)$. Since $l=2$ and $q=1$, so in view of Lemma \ref{biku1}, we get $H\equiv 0$.
Next we apply the Lemma \ref{biku}, and we obtain our desired result $f^{(k)}\equiv g^{(k)}$ when $n\geq 5$. Hence the theorem is proved.
\end{proof}
\begin{center} {\bf Acknowledgement} \end{center}
The authors wish to thank the referee for his/her valuable remarks and suggestions to-wards the improvement of the paper.\par
This research work is supported by the Council Of Scientific and Industrial Research, Extramural
Research Division, CSIR Complex, Pusa, New Delhi-110012, India, under the sanction project no. 25(0229)/14/EMR-II and Department of Science and Technology, Govt.of India under the sanction order no. DST/INSPIRE Fellowship/2014/IF140903.


\end{document}